\documentclass[11pt]{amsart}
\usepackage[all,cmtip]{xy}            
\usepackage{comment}
\usepackage{graphicx}
\usepackage{amssymb}
\usepackage{epstopdf}
\excludecomment{va}

\title{The Twisted K\"ahler-Ricci Flow}
\theoremstyle{plain}
\newtheorem{thm}{Theorem}
\newtheorem{prop}{Proposition}[section]
\newtheorem{defn}{Definition}[section]
\newtheorem{lem}{Lemma}[section]
\newtheorem{cor}{Corollary}

\theoremstyle{remark}
\newtheorem*{rem}{Remark}

\newcommand{\Tr}{\textrm{Tr}}

\newcommand{\del}{\partial} 
\newcommand{\dbar}{\overline{\del}}
\newcommand{\ddb}{i\del\dbar}
\newcommand{\osc}{\mathrm{osc}\,}
\newcommand{\ddt}{\frac{\del}{\del t}}

\newcommand{\R}{\mathbb{R}}

\newcommand{\W}{\mathcal{W}}
\renewcommand{\leq}{\leqslant}
\renewcommand{\geq}{\geqslant}

\author{Tristan C. Collins}
\address{Department of Mathematics, Columbia University, New York, NY}
\email{tcollins@math.columbia.edu}
\author{G\'abor Sz\'ekelyhidi}
\address{Department of Mathematics, University of Notre Dame, South Bend, IN}
\email{gszekely@nd.edu}
\begin{document}

\begin{abstract}
	In this paper we study a generalization of the K\"ahler-Ricci flow,
	in which the Ricci form is twisted by a closed, non-negative 
	(1,1)-form. We show that when a twisted K\"ahler-Einstein metric
	exists, then this twisted flow converges exponentially.  This
	generalizes a result of Perelman on the convergence of
	the K\"ahler-Ricci flow, and it builds on work of Tian-Zhu. 
\end{abstract}

\maketitle 
\section{Introduction}
The K\"ahler-Ricci flow, introduced by Hamilton \cite{Ham} has been studied
extensively in recent years. 
In this paper we study a generalization of the K\"ahler-Ricci flow, which
we call the \emph{twisted} K\"ahler-Ricci flow.  Fix a compact K\"ahler
manifold $(M,J)$.
\begin{defn}
Let $\alpha$ be a closed, non-negative $(1,1)$-form on $(M,J)$,
and suppose that $2\pi c_{1}(M)-\alpha >0$ is a K\"ahler class.  Fix
$\omega_{0} \in  2\pi c_{1}(M)-\alpha$.  The \emph{normalized} twisted
K\"ahler-Ricci flow is the
evolution equation
\begin{equation}\label{eq:twistedKRF}
	\ddt \omega = \omega + \alpha - \mathrm{Ric}(\omega),
\end{equation}
with $\omega(0)=\omega_0$ at $t=0$. 
\end{defn}

The normalized twisted K\"ahler-Ricci flow preserves the cohomology class
of $\omega$, and the short and long-time existence results follow from the
standard arguments for the K\"ahler-Ricci flow on a Fano manifold 
(see e.g. Cao \cite{Cao}, or the book \cite{C}).  Below we will also consider an unnormalized version of the flow. 
When  $\alpha=0$, both of these flows reduce to the usual
K\"ahler-Ricci flow on the underlying Fano manifold $(M,J)$. Our main
object of study in this paper is the convergence of the flow
\eqref{eq:twistedKRF} when a solution of the equation
\begin{equation}\label{eq:tKEmetric}
	\mathrm{Ric}(\omega) = \omega + \alpha
\end{equation}
exists. Solutions of this equation are called twisted K\"ahler-Einstein
metrics, and they arise in various settings, for instance in
Fine \cite{Fine} and Song-Tian \cite{SongTian}.  Of particular interest recently
has been the generalization where $\alpha$ is a multiple of the
current of integration along a divisor, in relation with K\"ahler-Einstein
metrics which have conical singularities (see for example Donaldson
\cite{Don}, Jeffres-Mazzeo-Rubinstein \cite{JMR}). In this paper, however, 
we will focus on the case when $\alpha$ is a smooth form. 

Our main result is the following.
\begin{thm}\label{thm:main}
	Suppose that there is a solution $\omega$ of
	equation~\eqref{eq:tKEmetric}. Then for any $\omega_0\in [
		\omega]$, the flow \eqref{eq:twistedKRF} with initial
		metric $\omega_0$ converges
		exponentially fast to a (perhaps different) solution of
		\eqref{eq:tKEmetric}. 
\end{thm}

Similar results can be proved in the case when 
$c_1(M) - \alpha\leqslant 0$, but this is essentially 
contained in the work of Cao \cite{Cao}.
In the case when $\alpha=0$, our main theorem reduces to an unpublished
result of Perelman.  Namely, we obtain as a corollary;

\begin{cor}
Suppose the Fano manifold $(M, J)$ admits a K\"ahler-Einstein metric.  Then for any $\omega_{0} \in c_{1}(M)$,
the K\"ahler-Ricci flow with initial metric $\omega_{0}$ converges exponentially fast to a K\"ahler-Einstein metric.
\end{cor}

This theorem has been addressed several times in the literature, most
notably by Tian-Zhu and collaborators (see \cite{TZ, TZ3, TZZZ}). 
Our approach in this paper is based on the ideas in \cite{TZ},
in particular we make strong use of the result of Tian-Zhu \cite{TZ} that
Perelman's entropy functional increases to a fixed topological constant
along the K\"ahler-Ricci flow.  

The first step in the proof of Theorem~\ref{thm:main} is an
extension of Perelman's estimates~\cite{ST} to the twisted
flow.  For our later applications, we require uniform control of the constants appearing
Perelman's estimates for a family of twisted K\"ahler-Ricci flows with initial metrics lying 
in a bounded family in $C^{3}$.  This requires us to reformulate the arguments in  \cite{ST} 
in order to obtain effective bounds. The presence of the 
extra form $\alpha$ causes little difficulty, although at various points it is important that $\alpha$ is
closed and non-negative.  In addition, we extend to the twisted case the uniform Sobolev inequality along 
the K\"ahler-Ricci flow, proved by Zhang \cite{QZ}.  These developments appear in 
Sections~\ref{sec: w func} and~\ref{sec: Perelman}.  With these results established, we show in 
Section~\ref{sec: Mabuchi} that Perelman's entropy functional increases along the twisted KRF to a
fixed topological constant, extending a result of Tian-Zhu \cite{TZ} to the
twisted setting.  

Finally,
in Section~\ref{sec: convergence} we prove Theorem~\ref{thm:main}.  The proof 
is by a method of continuity argument for the initial metric, similar to
the method  of Tian-Zhu in \cite{TZ}.  
The main difference between the two
approaches is the use of different norms to measure the distance from a
metric to a K\"ahler-Einstein metric. In \cite{TZ}, the distance between
two metrics is measured by setting
\begin{equation*}
||g-g'||_{C^{\ell}(M)} = \inf_{\Phi}|g- \Phi^{*}(g')|_{C^{\ell}(M)},
\end{equation*}
where the norm on the right hand side is computed with respect to a
fixed metric, and the infimum runs over all diffeomorphisms of $M$.  
In this paper, we measure the distance
between  an evolving metric and a twisted K\"ahler-Einstein metric, using
instead a $C^{0}$-norm on K\"ahler potentials.  More precisely, fixing a
twisted K\"ahler-Einstein metric $g_{tKE}$, we can write $\tau^*g = g_{tKE}
+ i\del\overline{\del}\phi_{\tau}$ for any biholomorphism $\tau$ of
$(M,J)$ which fixes $\alpha$.  Then, we set
\begin{equation*}
d(g) = \inf_{\tau} \text{osc} \phi_{\tau}.
\end{equation*}
We believe that using this norm
makes some of the arguments more transparent.
Moreover, working
with K\"ahler potentials allows us to avoid using a
result analogous to Chen-Sun's generalized uniqueness theorem ~\cite{CS}
similarly to Tian-Zhang-Zhang-Zhu~\cite{TZZZ}.
Instead, we only need the extension of Bando-Mabuchi's result \cite{BM} to
the twisted case, which was given by Berndtsson \cite{Ber}.

Before proceeding, we make a note about conventions.  In Section~\ref{sec:
w func} we work in the Riemannian setting. 
We take this approach, since our results hold in the case of the \emph{real} Ricci flow, twisted by a $(0,2)$ tensor satisfying
a ``contracted Bianchi identity".  In particular, all geometric quantities are the \emph{Riemannian} quantities.  In all subsequent sections, we work in complex coordinates, with the corresponding complex quantities.

\subsection*{Acknowledgements}
We would like to thank Professor D. H. Phong for his interest in this work
and his encouragement. We are also grateful to Valentino Tosatti and Ben
Weinkove for helpful comments, and for pointing out several typos in a previous
version of this paper .

\section{The Twisted $\W$-functional}\label{sec: w func}

In this section we introduce the twisted analog of some of Perelman's
functionals.  These functionals will play a crucial role in our later
estimates. For this section only, we will consider the 
\emph{unnormalized} twisted
K\"ahler-Ricci flow, which is the evolution equation
\begin{equation}\label{eq:unnorm}
\ddt \omega = -2(\mathrm{Ric}(\omega) -\alpha).
\end{equation}
This will alow us to use calculations in the existing literature more
readily. Note that if $\omega(t)$ is a solution of the normalized flow
\eqref{eq:twistedKRF}, then $\tilde{\omega}(t) = (1-2t)\omega(-\log(1-2t))$ is a
solution of the unnormalized flow with the same initial condition. 
In particular in our situation the existence time of the flow
\eqref{eq:unnorm} is $t\in[0,\frac{1}{2})$. 

\begin{defn}
Let $(M,g,J)$ be a compact K\"ahler manifold of complex dimension $n$, and
let $\alpha$ be a closed, non-negative $(1,1)$-form.  Let $\beta:= \alpha(\cdot, J\cdot)$ be
the induced, symmetric, non-negative $(0,2)$ tensor.  Define the twisted
entropy functional $\W: \mathfrak{Met} \times C^{\infty}(\R) \times \R_{>0}
\rightarrow \R$ by
\begin{equation*}
\W(g,f,\tau) := (4 \pi \tau)^{-n} \int_{M} \left( \tau(R-\Tr_{g}\beta + |\nabla f|^{2}) + (f-2n)\right)e^{-f}dm
\end{equation*}
where $dm = \sqrt{\det g}$ denotes the Riemannian volume form of $g$, and all quantities are the real quantities.
\end{defn}

\begin{thm}\label{W monotonicity}
Suppose the $(g(t), f(t), \tau(t)) \in \mathfrak{Met} \times C^{\infty}(\R) \times \R_{>0}$ solves the coupled system of partial differential equations
\begin{equation}\label{tKRF}
\ddt g = -2(Ric(g)- \beta)
\end{equation}
\begin{equation}\label{f evolution}
\ddt f = -\Delta_{g} f + |\nabla f |_{g}^{2} -R(g) + \Tr_{g}\beta + \frac{n}{\tau},
\end{equation}
\begin{equation}\label{tau evolution}
\frac{d}{dt} \tau = -1,
\end{equation}
on the interval $[0,T]$. Then, $\W(t) := \W(g(t),f(t),\tau(t))$ satisfies
\begin{equation*}
\begin{aligned}
\frac{d}{dt} \W(t) &= \\ &2\tau
\int_{M} \left( \left|Ric(g) +\nabla\nabla f-\beta
-\frac{g}{2\tau}\right|^{2}_{g} + \beta(\nabla f, \nabla f) \right) (4\pi\tau)^{-n}e^{-f} dm.
\end{aligned}
\end{equation*}
In particular, $\W(g(t),f(t),\tau(t))$ is monotonically increasing in $t$. 
\end{thm} 

\begin{proof}
For the proof, we work in \emph{real} coordinates.  From now on, we shall suppress the dependence on $g$, with the understanding that all Laplacians, curvatures, traces and inner products are computed with respect to $g$, unless otherwise noted.  Denote by $\W(t) := \W(g(t), f(t), \tau(t))$.  Following the computation of the variational formula for Perelman's entropy functional \cite{CZ, C} we find the variational formula for the twisted entropy functional is given by
\begin{equation}\label{ddtW}
\begin{aligned}
(4\pi\tau)^{n}&\frac{d}{dt} \W(t) \\&= \int_{M}-\tau\left\langle \ddt g, Ric - \beta +\nabla \nabla f -\frac{1}{2\tau}g\right\rangle e^{-f}dm\\
& + \int_{M} \left(\frac{1}{2}\Tr \ddt g - \ddt f +\frac{n}{\tau} \right)
	\Big[ \tau(R - \Tr \beta + 2 \Delta f - |\nabla f|^{2}) \\
&\qquad\qquad\qquad\qquad\qquad\qquad\qquad + f - 2n-1\Big]
		e^{-f}dm \\
& - \int_{M} \left( R -\Tr \beta+ |\nabla f|^{2} - \frac{n}{\tau} \right)
e^{-f}dm
\end{aligned}
\end{equation}
Plugging in the evolution equations~\eqref{tKRF},~\eqref{f evolution},~\eqref{tau evolution}, we obtain 
\begin{equation*}
\begin{aligned}
(4\pi\tau)^{n}&\frac{d}{dt} \W(t) \\
&= \int_{M}-2\tau\left\langle Ric-\beta, Ric - \beta +\nabla \nabla f -\frac{1}{2\tau}g\right\rangle e^{-f}dm\\
& + \int_{M}(\Delta f - |\nabla f|^{2})(\tau(R- \Tr \beta + 2\Delta f - |\nabla f|^{2}) + f] e^{-f}dm \\
& - \int_{M} \left( R -\Tr \beta+ |\nabla f|^{2} - \frac{n}{\tau} \right) e^{-f}dm
\end{aligned}
\end{equation*}
where in the second line we have used that $\int_{M}( \Delta f- |\nabla
f|^{2}) e^{-f}dm =0$. Since $\alpha$ is a closed, $(1,1)$-form, the tensor $\beta$ 
 satisfies the ``contracted Bianchi identity"
\begin{equation}\label{bianchi}
\nabla_{i}\Tr \beta = 2g^{jp}\nabla_{p}\beta_{ij}.
\end{equation}
Using this identity, the second term can be manipulated as follows;
\begin{equation*}
\begin{aligned}
 &\int_{M}(\Delta f - |\nabla f|^{2})(\tau(R- \Tr \beta + 2\Delta f - |\nabla f|^{2}) + f] e^{-f}dm\\
 &= \int_{M}(\Delta f - |\nabla f|^{2})(2\tau \Delta f - \tau |\nabla f|^{2})e^{-f}dm\\
 &- \int_{M}|\nabla f|^{2} e^{-f}dm - \tau \int_{M} \langle \nabla f, \nabla (R-\Tr \beta)\rangle e^{-f}dm \\
 &=\tau\int_{M} \langle -\nabla f, \nabla (2\Delta f - |\nabla f|^{2})\rangle e^{-f}dm\\
 &-\int_{M}\Delta f e^{-f}dm -2\tau \int_{M} g^{jp}g^{ki}(\nabla_{p}R_{ij}-\nabla_{p}\beta_{ij})\nabla_{k}f e^{-f}dm\\
   \end{aligned}
 \end{equation*}
 \begin{equation*}
\begin{aligned}
 &=-2\tau \int_{M} g^{ij}\nabla_{i}f(\nabla_{j}\Delta f -\langle \nabla f,\nabla_{j}\nabla f \rangle)e^{-f}dm\\
 &+2\tau \int_{M}g^{jp}g^{ki}[(R_{ij}-\beta_{ij})\nabla_{p}\nabla_{k}f  - \nabla_{p}f\nabla_{k}f (R_{ij}-\beta_{ij})]e^{-f}dm\\
 &+2\tau \int_{M}\left\langle \frac{g}{2\tau}, \nabla \nabla f \right\rangle e^{-f}dm\\
 &= 2\tau \int_{M}\left[\left \langle \nabla\nabla f , Ric -\beta+\nabla \nabla f - \frac{1}{2\tau}g\right\rangle + \beta(\nabla f, \nabla f) \right] e^{-f} dm  .
 \end{aligned}
 \end{equation*}
Moreover, the third term can be written as
\begin{equation*}
\begin{aligned}
&-\int_{M} \left( R -\Tr \beta+ |\nabla f|^{2} - \frac{n}{\tau} \right) e^{-f}dm \\
&= 2\tau \int_{M}\left \langle \frac{-1}{2\tau}g, Ric-\beta +\nabla\nabla f - \frac{1}{2\tau}g\right\rangle e^{-f}dm.
\end{aligned}
\end{equation*}
Combining these three expressions we obtain
\begin{equation*}
\frac{d}{dt} \W(t) = 2\tau \int_{M} \left(\left|Ric - \beta +\nabla\nabla f - \frac{1}{2\tau}g \right|_{g}^{2} + \beta(\nabla f, \nabla f)\right)(4\pi\tau)^{-n}e^{-f}dm,
\end{equation*}
which proves the theorem.

\end{proof}

\begin{rem}
Note that this computation works for any \emph{real} Ricci flow twisted by
any $(0,2)$-tensor $\beta$ satisfying the contracted Bianchi
identity~(\ref{bianchi}) with respect to all metrics along the flow.  However, monotonicity 
may fail if $\beta$ is not semi-positive. 
\end{rem}

We define the twisted $\mu$ functional as follows;

\begin{defn}
The functional $\mu : \mathfrak{Met} \times \R_{>0} \rightarrow \R$ is defined by
\begin{equation*}
\mu(g,\tau) := \inf \left\{ \W(g,f,\tau) : f\in C^{\infty}(M, \R) \text{ satisfies } (g,f,\tau) \in \mathcal{X} \right\}
\end{equation*} 
where
\begin{equation*}
\mathcal{X} := \left\{(g,f,\tau) : \int_{M} (4\pi \tau)^{-n}e^{-f}dm =1 \right\}
\end{equation*}
\end{defn}

Many of the properties of the twisted $\mu$ functional carry over from the
standard Ricci flow.  We list them, without proof, in the following
proposition.  For details, we refer the reader to \cite{C}.

\begin{prop}\label{mu properties}
The twisted $\mu$ functional has the following properties;
\begin{itemize}
\item[({\it i})]
For any $c \in \R_{>0}$ we have $\mu(cg,c\tau) = \mu (g,\tau)$.
\item[({\it ii})]
For a fixed $\tau \in \R_{>0}$, the function $\mu(g,\tau)$ is continuous in the $C^{2}$ topology on $\mathfrak{Met}$.
\item[({\it iii})]
For any $(g,\tau)$ there exists a function $f \in C^{\infty}(M,\R) \cap \mathcal{X}$ such that $\W(g,f,\tau) = \mu(g,\tau)$.  In particular, $\mu(g,\tau) >-\infty$.  Moreover, $f$ satisfies the nonlinear elliptic equation
\begin{equation}\label{min equation}
\tau(2\Delta f - |\nabla f|^{2} +R -\Tr_{g}\beta) +f -2n = \mu(g,\tau)
\end{equation} 
\item[({\it iv})]
Let $(g,\tau) \in \mathfrak{Met}\times \R_{>0}$ satisfy
equations~(\ref{tKRF}) and (\ref{tau evolution}) on $[0,T]$.  Then
functional $\mu (g,\tau)$ is monotonically increasing.  More precisely, for
any $t_{0} \in [0,T]$, if $f(t_{0}) \in \mathcal{X}$ is the minimizer whose
existence is guaranteed by ({\it iii}) we have
\begin{equation*}
	\begin{aligned}
		\frac{d}{dt}&\big|_{t=t_{0}}\mu(g,\tau)\\
	&\geq 2\tau(t_{0})
\int_{M}\left|Ric(t_{0}) + \nabla \nabla f(t_{0}) -\beta
-\frac{g(t_{0})}{2\tau(t_{0})}\right|^{2}
(4\pi\tau(t_{0}))^{-n}e^{-f(t_{0})}dm_{t_{0}}
\end{aligned}
\end{equation*}
in the sense of $\liminf$ backwards difference quotients.  In particular, for $r>0$ we have
\begin{equation*}
\mu(g(t_{0}), r^{2}) \geq \mu(g(0), r^{2}+t_{0})
\end{equation*}
which follows by taking $\tau(t) = t_{0}+r^{2}-t$.
\end{itemize}
\end{prop}

It follows from Proposition~\ref{mu properties} ({\it i}) and ({\it iv}),
and the relation between the normalized and unnormalized twisted K\"ahler
Ricci flow, that $\mu(g, \frac{1}{2})$ is monotonically increasing along
the normalized flow.  For convenience, we record this in the following
lemma.

\begin{lem}\label{lem: mu upper bound}
The quantity $\mu(g,\frac{1}{2})$ is monotonically increasing along the
normalized twisted K\"ahler-Ricci flow.  Moreover, $\mu(g,\frac{1}{2})$ is
uniformly bounded above by a topogical constant,
\begin{equation*}
\mu(g,\frac{1}{2}) \leq  \log\left[(2\pi)^{-n}Vol(M)\right].
\end{equation*}
\end{lem}

\begin{proof}
This follows immediately by substituting 
\[ f= -n\log(2\pi) + \log( Vol(M) ),\]
into the $\mathcal{W}$-functional, 
using the relation between the \emph{real} scalar curvature and the complex
scalar curvature, and the semi-positivity of $\beta$.
\end{proof}

In the remainder of this section we prove a non-collapsing estimate along
the twisted K\"ahler-Ricci flow,
extending Perelman's estimates for the K\"ahler-Ricci flow.  In the untwisted case, this result
is an improvement (due to Perelman), of Perelman's original non-collapsing result \cite{KleinerLott, ST}.  
For our later applications, it will be important to prove \emph{effective}
estimates, with clear dependence on the geometry of the initial metric
$g_{0}$.  First, we need the following easy lemma.

\begin{lem}\label{radius Lemma}
Fix $x \in M$ and $t \in [0,\frac{1}{2})$, and suppose there is an $ r > 0 $ such that $|R(g(t))-Tr_{g(t)}\beta| \leq \frac{K}{r^{2}}$ on $B(x,r)$.  Then there exists $r' \in (0,r]$ such that
\begin{enumerate}
\item[{\it (i)}] $|R - Tr_{g}\beta| \leq \frac{K}{r'^{2}}$ on $B(x,r')$
\item[{\it (ii)}] $(r')^{-2n}Vol(B(x,r'))\leq r^{-2n}Vol(B(x,r)) $
\item[{\it (iii)}] $Vol(B(x,r'))\leq 3^{2n}Vol(B(x,r'/2))$.
\end{enumerate}
\end{lem}
\begin{proof}
	The first item holds for any $r' \leqslant r$. 
By the standard expansion of the volume of a geodesic ball we have
\begin{equation*}
\lim_{k\rightarrow \infty} \frac{Vol(B(x,r/2^{k})) }{Vol(B(x,r/2^{k+1}))} = 2^{2n}
\end{equation*}
Hence, there is a $k < \infty$ such that
\begin{equation}\label{eq:iii}
\frac{Vol(B(x,r/2^{k}))}{Vol(B(x,r/2^{k+1}))}\leq 3^{2n},
\end{equation}
and if $l < k$, then
\begin{equation}\label{iterate}
 \frac{Vol(B(x,r/2^{l})) }{Vol(B(x,r/2^{l+1}))}>3^{2n}.
 \end{equation}
Choosing $r'=r/2^k$, item (iii) follows from \eqref{eq:iii}, while item
(ii) follows by 
iterating the inequality (\ref{iterate}) 
\end{proof}
 
We first prove the non-collapsing estimate along the \emph{unnormalized}
twisted K\"ahler-Ricci flow.
The corresponding estimate along the tKRF is then obtained by rescaling.

\begin{prop}\label{non collapsing unnorm}
Let $\tilde{g}(s)$ be a solution of the unnormalized twisted K\"ahler-Ricci flow with $\tilde{g}(0)=g_{0}$.  
Fix a number $\rho >0$, and define
\begin{equation*}
A(\rho):= \inf_{\tau \in [0,\frac{1}{2}+\rho^{2}]} \mu(g(0),\tau) >-\infty.
\end{equation*}
Then, for all  $(x,t) \in M\times [0,\frac{1}{2})$ and  $0<r \leq \rho$ 
such that 
\[ r^{2}|R(\tilde{g}(s))-\Tr_{\tilde{g}(s)}\beta| \leq K \text{ on } 
	B_{\tilde{g}(s)}(x,r),\]
there holds 
\begin{equation*}
Vol_{\tilde{g}(s)}(B(x, r)) \geq \kappa(K,\rho) r^{2n},
\end{equation*}
where $\kappa(K, \rho) = \exp
\left(A(\rho)+2n+n\log(4\pi)-3^{2n+2}-K\right)$.
\end{prop}
\begin{proof}
For simplicity, we suppress the dependence on $\tilde{g}(s)$. 
Fix a point $x \in M$, a radius $r \in (0, \rho]$ 
and a time $t \in [0,1/2)$. Suppose that  
\begin{equation*}
Vol(B(x, r)) < \kappa r^{2n}.
\end{equation*}
Let $r'\in (0, r]$ be the number provided by  Lemma~\ref{radius Lemma}.  
Then we have \\$r'{}^{2}|R-\Tr \beta| \leq K$ on $B(x,r')$, and
\begin{equation*}
Vol(B(x, r')) < \kappa r'{}^{2n}.
\end{equation*}
In order to ease notation, we now set $r=r'$.  Let $\phi: [0,\infty)\rightarrow \mathbb{R}$ 
be the function which is $1$ on $[0, 1/2)$, decreases linearly 
to zero on $(1/2,1]$, and is identically zero on $[1,\infty)$.  Then we set
\begin{equation*}
u(y) = e^{C}\phi(r^{-1}d(x,y))
\end{equation*}
where $C$ is chosen so that
\begin{equation*}
(4\pi)^{n} = e^{2C}r^{-2n}\int_{B(x,r)}\phi(r^{-1}d(x,y))^{2}dm(y).
\end{equation*}
It follows immediately from the definition of $\phi$ that 
\begin{equation*}
C > \frac{n}{2}\log(4\pi) - \frac{1}{2}\log(\kappa).
\end{equation*}
We now plug the function $u$ into the twisted entropy functional to get
\begin{equation*}
\begin{aligned}
\W(g(t),u, r^{2}) &\leq 8(4\pi)^{-n}r^{-2n}e^{2C}\left[Vol(B(x, r)) -
	Vol(B(x, r/2))\right]\\
	&\quad + K -2n -2C\\
&\leq 3^{2n+2}(4\pi)^{-n}r^{-2n}e^{2C}Vol(B(x, r/2)) + K-2n-2C\\
&\leq 3^{2n+2}(4\pi)^{-n}r^{-2n}e^{2C}\int_{M}\phi^{2}(y)dm(y)+ K-2n-2C\\
&< 3^{2n+2} + K - 2n - n\log(4\pi) + \log(\kappa).
\end{aligned}
\end{equation*}
Since $A(\rho) \leq \mu(g(0),t+r^{2}) \leq \mu(g(t), r^{2})$ for any $r\in [0,\rho)$, we obtain
\begin{equation*}
A(\rho) < 3^{2n+2} + K - 2n - n\log(4\pi) + \log(\kappa),
\end{equation*} 
from which it follows that $\kappa > \exp(A(\rho) + 2n + n\log(4\pi) -
3^{2n+2}-K)$. 
\end{proof}

Along the unnormalized flow we obtain;
\begin{prop}\label{non collapsing}
Let $g(t)$ be a solution of the normalized twisted K\"ahler-Ricci flow.
Fix a number $\rho >0$.  Then, for all  $(x,t) \in M\times [0,\infty)$
and  $0<r \leq e^{t/2}\rho$ 
such that $r^{2}|R(g(t))-\Tr_{g(t)}\beta| \leq K$ 
on $B(x,r)$ there holds 
\begin{equation*}
Vol_{g(t)}(B(x, r)) \geq \kappa(K,\rho) r^{2n}
\end{equation*}
where $\kappa(K, \rho)$ is defined in Proposition~\ref{non collapsing unnorm}.
\end{prop}
\begin{proof}
The proof is just an exercise in scaling.  We include the details for completeness.  Define a solution of the unnormalized tKRF by setting $\tilde{g}(s) = (1-2s)g(t(s))$ where $t(s) = -\ln(1-2s)$.
Set $\tilde{r} = e^{-t/2}r$.  Then $0 < r \leq \rho$, and we have
\begin{equation*}
\tilde{r}^{2}|R(\tilde{g}) - \Tr_{\tilde{g}}\beta| = r^{2}|R(g) - \Tr_{g}\beta| \leq K
\end{equation*}
on $B_{\tilde{g}}(x, \tilde{r}) = B_{g}(x,r)$.  Thus, we can apply the non-collapsing estimate of Proposition~\ref{non collapsing unnorm} to $\tilde{g}(s)$ to obtain
\begin{equation*}
Vol_{\tilde{g}}(B(x,\tilde{r})) \geq \kappa(K, \rho) \tilde{r}^{2n}.
\end{equation*}
 Replacing $\tilde{g}$ with $g$ proves the proposition.
 \end{proof}

\section{Perelman type estimates and the Sobolev inequality}\label{sec: Perelman}
In this section we extend Perelman's bounds \cite{ST}
on the scalar curvature and
diameter along the K\"ahler-Ricci flow to the twisted flow. We will need
effective estimates in our application, so we are careful to keep track of
constants. In addition one difference with the arguments in \cite{ST} is
that we bound $u$ independently of the diameter. This may be of independent
interest in situations when the diameter is not bounded. 

Since the 
twisted K\"ahler-Ricci flow (tKRF) preserves the cohomology class of
$\omega$, we can write it in terms of the K\"ahler potential $\phi$
defined by $\omega(t) = \omega_{0}+\ddb \phi(t)$.  Let $u(t)$ be the
twisted Ricci potential defined by
\[ \ddb u = \omega + \alpha - \mathrm{Ric}(\omega). \]
Then on the level of potentials the twisted K\"ahler-Ricci flow is given by
\begin{equation}\label{eq:ddtphi}
\dot{\phi} = \log\left(\frac{(\omega_{0}+
\ddb\phi)^{n}}{\omega_{0}^{n}}\right) + \phi + u(0),
\end{equation}
up to the addition of a time dependent constant. At each time we can
normalize the twisted Ricci potential $u$, so that
\[ \int_M e^{-u}\,dm = \int_Mdm = Vol(M). \]
This normalization implies that $u$ must vanish somewhere. 
We can then normalize the potentials along the flow by setting $\phi(0)=0$,
and 
\[ \ddt \phi = u, \]
which gives an equation equal to \eqref{eq:ddtphi}, up to adding a time
dependent constant.

Differentiating \eqref{eq:ddtphi}, the evolution of $u$ is given by
\begin{equation}\label{u evolve}
 \ddt u = \Delta u + u - c(t),
 \end{equation}
where $c(t)$ is a time-dependent constant. We can compute $c$ from the
normalization condition, since
\[
	0 = \frac{d}{dt}\int_M e^{-u}dm = \int_M (-\Delta u - u + c +
\Delta u)e^{-u}\,dm, 
\]
so we need 
\[ c = \int_M ue^{-u}\,dm \leq0. \]
where the inequality follows from Jensen's inequality.

The following is analogous to the weighted Poincar\'e inequality of
Futaki \cite{Fut}(see also, \cite{TZ3}). The proof is identical, using that $\alpha$ is a
non-negative form. 
\begin{lem}\label{poincare}
For any $f$ on $M$ we have
\[ \frac{1}{V}\int_M f^2e^{-u}\,dm \leqslant \frac{1}{V}
\int_M |\nabla
f|^2e^{-u}\,dm + \left(\frac{1}{V}\int_M fe^{-u}\,dm\right)^2.\]
\end{lem}
A consequence is the following monotonicity.
\begin{lem}
\[
\frac{d}{dt} c =  \frac{d}{dt}\int_M ue^{-u}\,dm \geqslant 0. 
\]
\end{lem}
\begin{proof}
	We simply compute
	\[ \begin{aligned}
		\frac{d}{dt} \int_M ue^{-u}\,dm &= \int_M (\Delta
		u + u -c -u\Delta u -u^2 + cu + u\Delta u)
		e^{-u}\,dm \\
		&= \int_M (\Delta u - u^2 + cu)e^{-u}\,dm \\
		&= \int_M |\nabla u|^2 e^{-u}\,dm - \int_M
		u^2e^{-u}\,dm + \frac{1}{V}\left(\int_M
		ue^{-u}\,dm\right)^2,
	\end{aligned}\]
	but this last expression is non-negative by the weighted
	Poincar\'e inequality applied to $f=u$. 
\end{proof}

Next we need to find evolution equations for $|\nabla u|^2$ and $\Delta
u$. Standard calculations give the following.  
\begin{lem}\label{ev equations}
	We have
	\[ \begin{aligned}
		\ddt |\nabla u|^2 &= \Delta |\nabla u|^2 + 
		|\nabla u|^2 - \alpha^{j\bar k}\nabla_ju\nabla_{\bar k}u
		- |\nabla\nabla u|^2 - |\nabla\overline{\nabla}u|^2 \\
		 \ddt \Delta u &= \Delta(\Delta u) + \Delta u -
		|\nabla\overline{\nabla}u|^2.
	\end{aligned}\]
\end{lem}

Before proceeding, we note the following lemma, which follows immediately
from the maximum principle applied to the second equation in the preceding lemma .
\begin{lem}\label{laplace upper bd}
Along the twisted K\"ahler-Ricci flow we have 
\begin{equation*}
R-\Tr_{g}\alpha \geq -\max(R-\Tr_{g}\alpha)^{-}(0).
\end{equation*}
In particular, $\Delta u \leq n + \max(R-\Tr_{g}\alpha)^{-}(0)$.
\end{lem}

\begin{lem}
	Along the twisted K\"ahler-Ricci flow we have the lower 
	bound $$u(t) >-n+c(0) - \max(R-\Tr_{g}\alpha)^{-}(0)).$$
\end{lem}
\begin{proof}
The argument of Perelman, as communicated by Sesum-Tian \cite{ST} carries over to prove that
 $u(t)>-B$ for some constant $B>0$. In order to obtain the effective estimate, we observe that 
by equation~\eqref{u evolve} and the monotonicity of $c(t)$ we have
\begin{equation*}
\ddt u = n -c(t)+\Tr_{g}\alpha - R + u  \leq n-c(0) + \max(R-\Tr_{g}\alpha)^{-}(0)+u.
\end{equation*}
Set $C=n-c(0) + \max(R-\Tr_{g}\alpha)^{-}(0)$, then it follows that, for any $t_{0}$
\begin{equation*}
u(t) < e^{t-t_{0}}(u(t_{0}) +C) - C
\end{equation*}
The lower bound for $u$ implies that, at $t_{0}$, we have
\begin{equation*}
u(t_{0}) \geq -C = -n+c(0) - \max(R-\Tr_{g}\alpha)^{-}(0).
\end{equation*}
Since $t_{0}$ was arbitrary, the lemma follows.
\end{proof}

Using the evolution equations in Lemma~\ref{ev equations} the proof of the following lemma is
identical to that in \cite{ST}, again using that $\alpha$ is non-negative.
The effective bounds can be read off directly from the proof. 
\begin{lem}\label{lem:ST1}
	Let $B= n - c(0) + \max_{M} (R- \Tr_{g}\alpha)^{-}(0)$.  Then we have
	\[ \begin{aligned}
		|\nabla u|^2 &< 200B(u+200B) \\
		|\Delta u| &< 200B(u+200B). 
	\end{aligned}\]
\end{lem}

Let us assume now that we have a constant $K$ such that
\[ |\Delta u|, |\nabla u|^2 < Ku,\text{ wherever } u > K. \]
We can take $K=400B$ with the $B$ as in the previous lemma. For any
two numbers $a < b$ define the set
\[ M(a,b) = \{x\in M\,|\,a < u(x) < b\}.\]
These sets will be used instead of the geodesic annuli in \cite{ST}. 

\begin{lem}
	There is a constant $\kappa_{1} > 0$ such that if $a > K$ and $b >
	a+2$ then
	\[ Vol(M(a,b)) > \kappa_{1} a^{-n}, \]
	as long as there is a point $x$ with $u(x)=a+1$. 
\end{lem}
\begin{proof}
	On the set $M(a,a+2)$ we have $|\nabla u| < \sqrt{K(a+2)}$, so
	$M(a,a+2)$ contains the ball of radius
	\[ \frac{1}{\sqrt{K(a+2)}} \]
	around the point $x$. At the same time, on this ball we have
	\[ |\Delta u| < K(a+2), \]
	so by non-collapsing estimate in Proposition~\ref{non collapsing} the volume of this ball is at least
	$\kappa(n+1, 1) [K(a+2)]^{-n}$. This in turn is at
	least $\kappa_{1} a^{-n}$ for some other constant $\kappa_{1}$.
\end{proof}

\begin{lem}\label{lem:vol}
	Let $0<\epsilon <1$ and $k > \max\{\log_2 \kappa_{1}^{-1/n},2\}$. 
	Suppose that 
	\[ Vol(M(2^{k},2^{10k}))<\epsilon,\]
	and $u(x) > 2^{10k}$ for some $x$. 
	Then there exist integers $k_1,k_2\in[k,10k]$ with 
	$k_2 > k_1+4$ such that
	\[ \begin{aligned}
		Vol(M(2^{k_1},2^{k_2})) &< \epsilon, \\
		Vol(M(2^{k_1+2},2^{k_2-2})) &> 2^{-3n}
		Vol(M(2^{k_1},2^{k_2})).
	\end{aligned}\]
\end{lem}
\begin{proof}
	The first condition is true for any $k_1,k_2\in[k,10k]$ 
	by hypothesis. We claim that for some integer $p\in[0,2k-1]$ we
	also have 
	\[ 
	Vol(M(2^{k+2p}, 2^{9k+2-2p})) <
	2^{3n}Vol(M(2^{k+2p+2},2^{9k-2p})),
	\]
	in which case we could choose $k_1 = k+2p$ and $k_2=9k+2-2p$. 
	Otherwise we would have 
	\[\begin{aligned}
		1 > \epsilon &> Vol(M(2^{k},2^{9k+2})) \geqslant
		2^{3n} Vol(M(2^{k+2},
	2^{9k})) \geqslant \ldots \\
	&\geqslant 2^{6nk} Vol(M(2^{5k},2^{5k+2})) \geqslant 2^{6nk}\kappa_{1}
	2^{-5nk} = 2^{nk}\kappa_{1}.
	\end{aligned}\]
	Since by our assumption $2^{nk}\kappa_{1} > 1$, we get a contradiction.
\end{proof}

\begin{lem} \label{lem:intR} If $k_2 > k_1+1$, then 
	we can find $r_1\in[2^{k_1},2^{k_1+1}]$ and $r_2\in
	[2^{k_2-1},2^{k_2}]$ such that
	\[ \int_{M(r_1,r_2)} -\Delta u\,dm < CV,\]
	for some fixed constant $C$, where
	\[ V = Vol(M(2^{k_1},2^{k_2})).\]
\end{lem}
\begin{proof}
	On $M(2^{k_1},2^{k_1+1})$ we have $|\nabla u|^2 < 2^{k_1+1}K$,
	so
	\[ \int_{M(2^{k_1},2^{k_1+1})} |\nabla u|^2\,dm < 2^{k_1+1}KV.\]
	By the coarea formula we have
	\[ \int_{M(2^{k_1},2^{k_1+1})} |\nabla u|^2\,dm =
	\int_{2^{k_1}}^{2^{k_1+1}}\int_{\{u=t\}} |\nabla u|\,dS\,dt.\]
	Note that by Sard's theorem, the level sets $\{u=t\}$ are 
	smooth for almost all values of $t$, so that the integral on the right hand side
	makes sense.  
	It follows that there exists $r_1\in[2^{k_1},2^{k_1+1}]$ such
	that $\{u=r_{1}\}$ is smooth and
	\[ \int_{\{u=r_1\}} |\nabla u|\,dS < 2KV.\]
	Similarly there exists $r_2\in[2^{k_2-1},2^{k_2}]$ such that $\{u=r_{2}\}$ is smooth and
	\[ \int_{\{u=r_2\}} |\nabla u|\,dS < 2KV.\]
	It follows that
	\[ \begin{aligned} \int_{M(r_1,r_2)} -\Delta u\,dm 
		&\leqslant \int_{\{u=r_1\}} |\nabla u|\,dS +
		\int_{\{u=r_2\}} |\nabla u|\,dS \\
		&\leqslant 4KV.
	\end{aligned}\]
\end{proof}

\begin{prop}
	There exists a constant $\epsilon > 0$ 
	such that if for some 
	\[ k > \max\{\log_2 \kappa_{1}^{-1/n},2\}\]
	we have
	\[ Vol(M(2^k,2^{10k})) < \epsilon, \]
	then $u < 2^{10k}$ everywhere. 
\end{prop}
\begin{proof}
Suppose by contradiction that $u$ takes on larger values, so we can
find numbers $k_1, k_2$ satisfying the
conclusions of Lemma~\ref{lem:vol}, and then also
$r_1,r_2$ using Lemma~\ref{lem:intR}.
Let us choose a cutoff function $\phi$ on $\mathbf{R}$ such that
\[ \phi = \left\{\begin{aligned} 1\quad&\mathrm{in}\quad [2^{k_1+2},
	2^{k_2-2}] \\
	0\quad &\mathrm{outside}\quad [2^{k_1+1},2^{k_2-1}].
\end{aligned}\right.\]
We can do this in such a way that on the interval $[2^{k_1+1},
2^{k_1+2}]$ we have
\[ |\phi'(x)| < \frac{2}{2^{k_1+1}} \leqslant \frac{4}{x},\]
and also on the interval $[2^{k_2-2},2^{k_2-1}]$ we have
\[ |\phi'(x)| < \frac{2}{2^{k_2-2}} \leqslant \frac{4}{x}.\]
So in sum we have $|\phi'(x)| < 4/x$ for all $x$.

Let us now set $f = A\phi(u)$ for some constant $A$, normalised so that
\[ \int_M f^2\,dm = (2\pi)^n,\]
and we use this as a test function in the $\mathcal{W}$ functional, with
$\tau=\frac{1}{2}$. We have
\[\begin{aligned}
	(2\pi)^n = \int_M f^2\,dm &> A^2 Vol(M(2^{k_1+2},2^{k_2-2})) \\
	&>
A^2 2^{-3n}Vol(M(2^{k_1},2^{k_2}))
=: A^22^{-3n}V,
\end{aligned}\]
so
\[ A^2V < 2^{3n}(2\pi)^n. \]
The monotonicity of $\mu\left(g, \frac{1}{2}\right)$ implies that 
\begin{equation}\label{eq:Per}
	\int_M (-\Delta u)f^2 + 4|\nabla f|^2 - f^2\ln f^2 - 2n f^2\,dm > -C_1
\end{equation}
for some constant $C_1$.  Letting $\hat{C} = \hat{C}(g_{0})$ be the constant 
in Lemma~\ref{laplace upper bd}, so that $-\Delta u + \hat{C} \geq 0$.  We have
\begin{equation*}
\begin{aligned}
\int_M (-\Delta u)f^2\,dm &\leq \int_M (-\Delta u + \hat{C})f^2\,dm\\
&< A^2\int_{M(r_1,r_2)} (-\Delta u+ \hat{C})\,dm\\
& < A^2(CV + \hat{C}V) < 2^{3n}(2\pi)^n (C+ \hat{C}) ,
\end{aligned}
\end{equation*}
with the constant $C$ from Lemma~\ref{lem:intR}. Also 
\[ |\nabla f|^2 = A^2|\phi'(u)||\nabla u| < A^2\frac{4}{u}\sqrt{Ku} <
4A^2\sqrt{K},\]
assuming $u>1$ on $M(r_1,r_2)$, ie. if $r_1>1$. So we get
\[ \int_M 4|\nabla f|^2\,dm < 16A^2\sqrt{K}V < 2^{3n+4}(2\pi)^n\sqrt{K}.\]
Also
\[ \begin{aligned}
	\int_M -f^2\ln f^2\,dm &= A^2\int_M -\phi(u)^2\ln\phi(u)^2\,dm
	- \ln
	A^2\int_M f^2\,dm \\
	&< A^2V - 2\ln A \\
	&< 2^{3n}(2\pi)^n - 2\ln A.
\end{aligned}\]
We used that $-1 < x\ln x \leqslant 0$ for $x\in[0,1]$.
In sum from (\ref{eq:Per}) we get
\[2^{3n}(2\pi)^n (C+ \hat{C}) + 2^{3n+4}(2\pi)^n\sqrt{K} + 2^{3n}(2\pi)^n - 2n - 2\ln A 
> -C_1.\]
This gives an upper bound $A < C_2$. At the same time
\[ (2\pi)^n = \int_M f^2\,dm < A^2 V,\]
so we get $V > (2\pi)^nA^{-2} > C_2^{-2}$. Note that $C_2$ only depends on the initial
metric of the flow via the entropy and the constant $\hat{C}$. We can therefore choose $\epsilon = C_2^{-2}$.  
\end{proof}

\begin{cor} Along the flow, $u$ is uniformly bounded. 
\end{cor}
\begin{proof}
	Fix a $k > \max\{\log_2 \kappa_{1}^{-1/n}, 2\}$. 
	For any integer $N>1$ we have
	\[ \sum_{i=1}^N Vol(M(2^{10^{i-1}k},2^{10^i k})) < V, \]
	where $V$ is the volume of $M$, so if $N > V/\epsilon$, then
	there is an $i \leqslant N$ for which 
	\[ Vol(M(2^{10^{i-1}k},2^{10^ik})) < \epsilon. \]
	Using the previous Proposition, we must therefore have
	\[ u < 2^{10^Nk}. \]
\end{proof}

Now, since $u$ is uniformly bounded above, we obtain uniform upper bounds
for $|\nabla u|$ and $|\Delta u|$ from Lemma \ref{lem:ST1}.  
We summarize our results in the following
\begin{thm}\label{thm:pman}
Suppose that $g(t)$ evolves along the twisted K\"ahler-Ricci flow with
$g(0) = g_{0}$, then there exists a constant $C$ depending continuously on
the $C^{3}$ norm of $g_{0}$ (and a uniform lower bound on $g_0$), 
such that
\begin{equation*}
|u| + |\nabla u|_{g(t)} + |\Delta_{g(t)}u| \leq C.
\end{equation*}
\end{thm}

Notice that our estimates \emph{did not} require a diameter bound.  In this
sense our approach differs from that in \cite{ST}. The diameter bound can now be obtained
via a simple covering argument.

\begin{lem} Along the twisted K\"ahler-Ricci flow the diameter of $M$ is
	uniformly bounded.  More precisely,
\begin{equation*}
	\text{diam}(M,\omega(t)) \leq \frac{2^{2n}Vol(M)}{\kappa(K, 1/2)},
\end{equation*}
where $K$ denotes the uniform upper bound of $|R - \Tr_{g}\alpha|$ along
the flow, and $\kappa(K, 1/2)$ is defined in Proposition~\ref{non
collapsing unnorm}
\end{lem}
\begin{proof}
Fix points $p_{1}, p_{2} \in M$ with $d(p_{1},p_{2})=R$.  Let
$\gamma:[0,R]\rightarrow M$ be a length minimizing geodesic connecting
$p_{1}$ to $p_{2}$.  Let $B_0, B_1,\ldots, B_{\lfloor R\rfloor}$ be balls
of radius $1/2$, centered at the points
$\gamma(0),\gamma(1),\ldots,\gamma(\lfloor R\rfloor)$. These balls are
disjoint since $\gamma$ is length minimizing. By the bound on $|R-
Tr_g\alpha|$ and Proposition \ref{non collapsing unnorm}, we have
\[ Vol(M)\geqslant \sum_{0}^{\lfloor R\rfloor} Vol(B_i, \omega(t))
	\geqslant R\cdot 2^{-2n}\kappa(K, 1/2).\]
It follows that
\[ R \leqslant \frac{2^{2n} Vol(M)}{\kappa(K,1/2)}. \]
\end{proof}

The above estimates allow us to deduce a uniform Sobolev inequality along
the tKRF.  Since the proof follows closely the arguments of \cite{H, RY,
QZ} we list only the key ingredients.  The main observation is that the
monotonicity of the $\mu$ functional implies the following uniform,
restricted log Sobolev inequality.

\begin{prop}\label{uni log sob prop}
Let $g(t)$ be a solution of the twisted K\"ahler-Ricci flow defined on
$[0,\infty)$, with $g(0)=g_{0}$.  Define
\begin{equation*}
C_{1} = C(n) + 4n\log C_{S}(M, g(0)) + 4 \frac{Vol(M)^{-n}}{ C_{S}(M,
g_{0})^{2}} +\max_{M}(R(0)-Tr_{g_{0}}\alpha)^{-} ,
\end{equation*}
where $C_{S}(M, g_{0})$ denotes the Sobolev constant of $(M,g_{0})$.  Then,
for all $\epsilon \in (0,2]$, $t\in [0,\infty)$ and $u \in W^{1,2}(M)$
	satisfying $||u||_{L^{2}(g(t))}=1$, we have
\begin{equation*}
\int_{M}v^{2}\log v^{2} dm(t) \leq \epsilon^{2}\int_{M}\left(|\nabla
v|^{2} + \frac{(R(t)- \Tr_{g(t)}\alpha)}{4}u^{2}\right) dm(t) -2n\log
\epsilon + C_{1}.
\end{equation*}
\end{prop}

There is a close relationship between log Sobolev inequalities and
estimates for certain heat kernels.  Heat kernel estimates are in turn
closely related to Sobolev inequalities.  In the case of the Ricci flow,
the uniform log Sobolev inequality of the above proposition implies a
uniform Sobolev inequality along the flow.  Of course, a similar estimates
hold in the case of the twisted K\"ahler-Ricci flow.  Since the techniques
are by now well documented, we omit the details and record only the result.
For details in the K\"ahler case, we refer the reader to  \cite{H, RY,
QZ}.
\begin{prop}\label{uni sob}
Let $g(t)$ be a solution of the twisted K\"ahler-Ricci flow defined on
$[0,\infty)$ , with $g(0)=g_{0}$.  Define
\begin{equation*}
C_{2} = \sup_{M}(R(0)-\Tr_{g_{0}}\alpha)^{-} + C_{1}
\end{equation*}
where $C_{1}$ is the constant defined in Proposition~\ref{uni log sob
prop}.  Then there are positive constants $C(n), \beta(n)$, depending only
on $n$, such that for all $u \in W^{1,2}(M)$ we have
\begin{equation*}
\begin{aligned}
&\left(\int_{M}u^{2n/(n-1)} dm(t)\right)^{(n-1)/n} \leq
C(n)e^{\beta(n)C_{2}} \int_{M}\left(|\nabla
u|^{2}+\frac{(R-Tr_{g}\alpha)}{4}u^{2}\right) dm\\ &\qquad\qquad+
C(n)e^{\beta(n)C_{2}}\left(1+\max
(R(0)-\Tr_{g_{0}}\alpha)^{-}\right)\int_{M}u^{2}dm(t).
\end{aligned}
\end{equation*}
\end{prop}

\section{The twisted Mabuchi energy}\label{sec: Mabuchi}

In this section we study the tKRF under the assumption that a twisted
K\"ahler-Einstein metric exists. More generally it is enough to assume that
the twisted Mabuchi energy is bounded below on the K\"ahler class $2\pi
c_{1}(M)-\alpha$.  First, we recall the definition of the twisted Mabuchi
energy

\begin{defn}[\cite{Stoppa} Definition 1.8]
Fix $\omega_{0} \in 2\pi c_{1}(M)-\alpha$.  For any K\"ahler form
$\omega_{\phi} = \omega_{0} + \ddb \phi$, we define the twisted Mabuchi
energy by its variation at $\phi$
\begin{equation*}
\delta\mathcal{M}_{\alpha}(\delta \phi) = -\int_{M} \delta \phi\left(R(\omega_{\phi}) - \Tr_{g_{\phi}}\alpha - n\right)\omega_{\phi}^{n} dt
\end{equation*}
where $\delta \phi \in C^{\infty}(M, \R)$. 
\end{defn}

The Mabuchi energy of a metric $\omega$ is obtained by fixing a base point
$\omega_{0}$ and integrating the variation along a path of metrics
connecting $\omega_{0}$ to $\omega$.  It is a basic fact that the result is
independent of the path chosen. The connection with twisted
K\"ahler-Einstein metrics is the following. 

\begin{thm}\label{MabuchiLB}
	Suppose that $\mathrm{Ric}(\omega) = \omega + \alpha$. Then the
	twisted Mabuchi energy is bounded below on the K\"ahler class $2\pi
	c_1(M) - \alpha$. 
\end{thm}

This is a generalization of the result of Bando-Mabuchi~\cite{BM}. In the
twisted case it follows from the results of Chen-Tian~\cite{CT} (see e.g.
Stoppa~\cite{Stoppa} or Sz\'ekelyhidi~\cite{Sz}). 
Along the twisted K\"ahler-Ricci flow
the Mabuchi energy takes on a particularly convenient form.

\begin{lem}
Suppose that $\omega(t)$ evolves along the twisted K\"ahler-Ricci flow with
$\omega(0)=\omega_{0}$.  Then  the Mabuchi energy, with base point
$\omega_{0}$ is given by
\begin{equation*}
\mathcal{M}_{\alpha}(\omega_{0}, \omega(t)) = -\int_{0}^{t} \int_{M} |\nabla u|^{2}(s) \omega(s)^{n} ds.
\end{equation*}
In particular, if the Mabuchi energy is bounded below, then
\begin{equation*}
\lim_{t\rightarrow \infty}\int_{M} |\nabla u|^{2}(t) \omega(t)^{n}= 0.
\end{equation*}
\end{lem}
\begin{proof}
The first assertion follows directly from computation, so we omit the
details.  For the second assertion, observe that since the Mabuchi energy
is bounded below, there exists times $t_{i} \in [i, i+1]$ such that 
\begin{equation*}
\lim_{i \rightarrow \infty}\int_{M} |\nabla u|^{2}(t_{i})
\omega(t_{i})^{n}= 0.
\end{equation*}
This extends to the full sequence by observing that using
Theorem~\ref{thm:pman}, we have the differential inequality
\begin{equation*}
\ddt Y(t) \leq C Y(t)
\end{equation*}
where $Y(t) = \int_{M} |\nabla u|^{2}(t) \omega(t)^{n}$ and $C$ is
independent of time.  For details in the untwisted case, see Phong-Sturm
\cite{PS}. 
\end{proof}

This estimate allows us to improve the estimates in Theorem~\ref{thm:pman}. 
\begin{prop}\label{prop: improved pman}
Suppose that the twisted Mabuchi energy is bounded below on $2\pi c_{1}(M) -\alpha$.  Then
\begin{equation*}
\lim_{t\rightarrow \infty} |u(t)| + |\nabla u(t)| + |\Delta u(t)| =0.
\end{equation*}
\end{prop}
\begin{proof}
This is identical to the KRF.  See, for instance Phong-Song-Sturm-Weinkove
\cite{PSSW}.
\end{proof}
We can now employ these estimate to study the behaviour of the twisted $\mu$ functional along the tKRF.

\begin{prop}\label{prop: f est}
Let $f_{t}$ be the function achieving $\mu(g(t),\frac{1}{2})$, where $g(t)$ is a solution of the twisted Kahler-Ricci flow with initial value $g(0) =g_{0}$.  Then the following estimates hold along the twisted K\"ahler-Ricci flow.
\begin{itemize}
\item[({\it i})]  
There exists a constant $C_{1} = C_{1}(g_{0})$ such that $\sup_{M}|f_{t}|  \leq  C_{1}$.
\item[({\it ii})]
There exists a subsequence of times $t_{i} \in [i, i+1]$ such that
\begin{equation*}
\lim_{i\rightarrow \infty}\left(\int_{M} |\nabla f_{t_{i}}|^{2}dm_{t_{i}}\right)^{2} + \lim_{i\rightarrow \infty}\int_{M}|\Delta f_{t_{i}}|^{2} dm_{t_{i}} =0.
\end{equation*}
\item[({\it iii})] Along the sequence $t_{i}$ we have
\begin{equation*}
\lim_{i\rightarrow \infty} \int_{M}f_{t_{i}}e^{-f_{t_{i}}}dm_{t_{i}} = (2\pi)^{n}\log\left((2\pi)^{-n}Vol(M)\right)
\end{equation*}
\end{itemize}

\end{prop}
\begin{proof}
The proof of (i)
follows \cite{TZ} closely, so we will only outline the argument. 
We begin by proving the first bound.  For ease of notation, we suppress the
dependence on $t$.  From \eqref{min equation} together with the bounds on
$R-\Tr_g\alpha$ and $\mu(g,1/2)$, the minimizer $f$ satisfies
\[ \Delta f - \frac{1}{2}|\nabla f|^2  < C - f, \]
from which we get
\[ \Delta e^{-f/2} = -\frac{1}{2}e^{-f/2}\left(
	\Delta f - \frac{1}{2}|\nabla f|^2\right) >
	-\frac{1}{2}(C-f)e^{-f/2}. \]
Letting $h = e^{-f/2}$ we then have a constant $C_\delta$ for any $\delta >
0$ such that
\[ \Delta h \geqslant -h^{1+\delta} - C_\delta.\]
Moser iteration, together with the normalization $\int h^2\, dm =
(2\pi)^n$ implies an upper bound for $h$, i.e. a lower bound
$f \geq -C_1$ for $f$. 

We turn our attention now to the upper bound. Define 
\[ E_{A} = \{x\in M : f(x) < A \}.\]
Using the bound on $u$ from Theorem \ref{thm:pman} and the normalization of
$f$, we have
\[ \int_M e^{-f} e^{-u}\,dm > C^{-1}, \]
for some $C$. Using the normalization $\int_M e^{-u}\,dm = V$ 
together with the lower bound $f\geqslant -C_1$, this implies
that there exists a sufficiently large $A$, and a $\delta > 0$ such that
\[ \int_{E_A} e^{-u}\,dm > \delta. \]

Using the bound on $u$ again, we have
\begin{equation}\label{ineq 1}
\begin{aligned}
	\int_{M}fe^{-u} dm &= \int_{E_A} fe^{-u}\,dm + \int_{M\setminus E_A}
	fe^{-u}\,dm \\
	&\leqslant AC + (V - \delta)^\frac{1}{2} \left(\int_M
	f^2e^{-u}\,dm\right)^\frac{1}{2}. 
\end{aligned}
\end{equation}
It follows that for some $C_2$ we have
\begin{equation}\label{eq:intfeu}
	\left(\int_M fe^{-u}\,dm\right)^2 \leqslant C_2 + \left(V -
	\delta/2\right)\int_M f^2e^{-u}\,dm.
\end{equation}
	Multiplying equation~(\ref{min equation}) by $e^{-u}$, integrating,
	and using the uniform bounds on $u$, 
$R-\Tr_{g}\alpha$ and $\mu(g)$, we obtain
\begin{equation*}
\begin{aligned}
	\int_{M}|\nabla f|^{2}e^{-u}\,dm &\leq \int_{M}f e^{-u}dm + C_3
	\\
	&\leq \frac{\delta}{4V}\int_M f^2e^{-u}\,dm + C_4,
\end{aligned}
\end{equation*}
for some $C_3, C_4$. 
Substituting this into the weighted Poincar\'e inequality of
Lemma~\ref{poincare}, and using \eqref{eq:intfeu} yields
\begin{equation*}
\begin{aligned}
\int_{M}f^{2}e^{-u}dm &\leq \int_{M}|\nabla f|^{2} e^{-u}dm + \frac{1}{V} \left(\int_{M}fe^{-u}dm\right)^{2}\\
&\leq \frac{\delta}{4V}\int_{M}f^2e^{-u}dm + \left(1 -
\frac{\delta}{2V}\right)\int_{M}f^2 e^{-u}dm + C_5.
\end{aligned}
\end{equation*}
Rearranging this, we get an upper bound
\begin{equation*}
\int_{M}f^{2}e^{-u} dm \leq C_6
\end{equation*}
where $C_6$ can be chosen to depend only on $g(0)$.  By equation~(\ref{min
equation}) we have $\Delta f \geq -f -C$, so 
the upper bound for $f$ follows from Moser
iteration and the $L^{2}$ bound. 

We now prove the second and third items.  We begin by observing that 
\begin{equation}\label{decay inequality}
\mu(g(T))-\mu(g(0)) \geq \int_{0}^{T} \left(\int_{M} \left|Ric(g) +\nabla\nabla f-\alpha -g\right|^{2}_{g}(2\pi)^{-n} e^{-f} dm\right)(s) ds
\end{equation}
To see this, we fix a partition $P_{N} = \{0=t_{0} < t_{1} <\dots <t_{N} =T\}$ of $[0,T]$, and write
\begin{equation*}
\mu(T)-\mu(0) = \sum_{i=1}^{N}\frac{\mu(t_{i})-\mu(t_{i-1})}{t_{i}-t_{i-1}}(t_{i}-t_{i-1}).
\end{equation*}
Let $f_{i}$ be the smooth function satisfying $\W(g_{i},f_{i},\frac{1}{2}) = \mu(t_{i})$, and let $f_{i}(t)$ be the solution to the backwards heat equation~(\ref{f evolution}) on $[t_{i-1},t_{i}]$ with $f_{i}(t_{i}) = f_{i}$.  Then by the mean value theorem we have
\begin{equation*}
	\begin{aligned}
\frac{\mu(t_{i})-\mu(t_{i-1})}{t_{i}-t_{i-1}} &\geq\frac{
	\W(g_{i},f_{i},\frac{1}{2})
	-\W(g_{i-1},f_{i}(t_{i-1}),\frac{1}{2})}{t_{i}-t_{i-1}}\\
	&=(t_{i}-t_{i-1})\frac{d}{dt}\bigg|_{t=t_{i}^{*}} 
	\W\left(g(t),f_{i}(t),\frac{1}{2}\right),
\end{aligned}
\end{equation*}
for some $t_{i}^{*} \in (t_{i-1},t_{i})$.  Using the result of the computation in the proof of Theorem~\ref{W monotonicity} we have
\begin{equation*}
\mu(T)-\mu(0) \geq \sum_{i=1}^{N}(t_{i}-t_{i-1})\left(\int_{M} \left|Ric(g) +\nabla\overline{\nabla} f_{i}-\alpha -g\right|^{2}_{g}  e^{-f_{i}} dm\right)(t_{i}^{*}).
\end{equation*}
Taking the $\liminf$ as $N\rightarrow \infty$ proves the result.  Since $\mu$ is increasing and bounded above, it follows immediately that that the bracketed term on the right had side of~(\ref{decay inequality}) goes to zero along a subsequence of times $t_{i} \in [i, i+1]$.  In particular, we have
\begin{equation*}
\lim_{i \rightarrow \infty} \int_{M} |\nabla\overline{\nabla} f_{i}  -\nabla\overline{\nabla} u(t_{i})|^{2}e^{-f_{i}}dm =0.
\end{equation*}
Now observe that $ |\nabla\overline{\nabla} f_{i}  -\nabla\overline{\nabla} u(t_{i})|^{2} \geq n^{-1}|\Delta(f-u)|^{2}$.   Applying Proposition~\ref{prop: improved pman}, we obtain
\begin{equation*}
\lim_{i \rightarrow \infty} \int_{M} |\Delta f_{i}|^{2} e^{-f_{i}}dm =0.
\end{equation*}
 The second item follows, using the upper bound for $f$, and the observation
\begin{equation*}
\int_{M}|\nabla f|^{2} dm  = -\int_{M}f \Delta f dm \leq C\left(\int_{M}|\Delta f|^{2}dm\right)^{1/2}
\end{equation*}
where $C$ depends only on the bound for $f$. 

Finally, we prove the third item. Here our argument differs somewhat from
\cite{TZ}. First, it follows from Jensen's inequality that 
\begin{equation*}
\int_{M}fe^{-f} \leq (2\pi)^{n}\log\left((2\pi)^{-n}V\right).
\end{equation*}
Thus, it suffices to prove a lower bound.  Set 
\begin{equation*}
\tilde{f} := f - V^{-1}\int_{M}(f-u)e^{-u} dm.
\end{equation*}
By the weighted Poincar\'e inequality of Lemma~\ref{poincare}, and the choice of normalization we have
\begin{equation*}
\begin{aligned}
\int_{M}(\tilde{f}-u)^{2}dm &\leq C \int_{M}(\tilde{f}-u)^{2}e^{-u}dm \\
&\leq C\int_{M}\left(|\nabla f|^{2} + |\nabla u|^{2}\right)e^{-u}dm\\
\end{aligned}
\end{equation*}
for a constant $C$ depending only on $g_{0}$.  From this and the upper bound for $|f|$ we obtain
\begin{equation*}
\begin{aligned}
\int_{M} |\tilde{f}-u|e^{-f}dm &\leq C \int_{M}|\tilde{f}-u|e^{-u}dm\\
& \leq C' \int_{M}(\tilde{f}-u)^{2}e^{-u}dm\\
&\leq C'' \int_{M}( |\nabla f|^{2} + |\nabla u|^{2}) e^{-u}dm.
\end{aligned}
\end{equation*}
In particular, applying item ({\it ii}), we have that
\begin{equation}\label{eq: f-u decay}
\lim_{i\rightarrow \infty} \int_{M} |\tilde{f_{i}}-u|e^{-f_{i}}dm =0.
\end{equation}
Moreover, by Jensen's inequality we have
\begin{equation*}
\begin{aligned}
\tilde{f}-f &= V^{-1}\int_{M}\log(e^{u-f})e^{-u}dm\\
&\leq \log\left( V^{-1}\int_{M}e^{-f}dm \right)\\
&= \log((2\pi)^{n}V^{-1}).
\end{aligned}
\end{equation*}
As a result, we have
\begin{equation*}
(2\pi)^{-n}\int_{M} (u-f)e^{-f}dm \leq \log((2\pi)^{n}V^{-1}) + \int_{M}(u-\tilde{f})e^{-f}dm
\end{equation*}
Rearranging this equation gives
\begin{equation*}
(2\pi)^{-n}\int_{M}fe^{-f} \geq -||u||_{C^{0}} + \log((2\pi)^{-n}V) -(2\pi)^{-n} \int_{M}|\tilde{f}-u|e^{-f}dm.
\end{equation*}
In particular, by Propostion~\ref{prop: improved pman} and  equation~\eqref{eq: f-u decay} we obtain that 
\begin{equation*}
\lim_{i\rightarrow \infty} \int_{M}f_{i}e^{-f_{i}} \geq (2\pi)^{n}\log((2\pi)^{-n}V),
\end{equation*}
which finishes the proof.
\end{proof}

As a result we obtain the following important corollary;

\begin{cor}\label{cor:mu max}
Suppose that the twisted Mabuchi energy of $(M,J)$ is bounded below on the
K\"ahler class $2\pi c_{1}(M)-\alpha$.  Let $\omega(t)$ be a solution of
the tKRF, with $\omega(0)=\omega_{0} \in 2\pi c_{1}(M)-\alpha$. Then
\begin{equation*}
\lim_{t\rightarrow \infty}\mu\left(\omega(t),\frac{1}{2}\right) =
\log\left((2\pi)^{-n}Vol(M)\right).
\end{equation*}
Moreover, if $\omega_0 \in 2\pi c_{1}(M)-\alpha$ satisfies 
\begin{equation*}
\mu(\omega_0, \frac{1}{2}) = \log\left((2\pi)^{-n}Vol(M)\right)
\end{equation*}
then $\omega_0$ is a twisted K\"ahler-Einstein metric.
\end{cor}
\begin{proof}
The first statement follows immediately from Proposition~\ref{prop: f est}.  We prove the second statement.  Let $\omega(t)$ be a solution of the tKRF with $\omega(0)= \omega_{0}$.  Then by the monotonicity of $\mu$, and Lemma~\ref{lem: mu upper bound} we have
\begin{equation*}
\log\left((2\pi)^{-n}Vol(M)\right) = \mu(\omega_{0}, \frac{1}{2}) \leq \mu(\omega(t), \frac{1}{2}) \leq \log\left((2\pi)^{-n}Vol(M)\right),
\end{equation*}
Thus, $\mu(t) = \log\left((2\pi)^{-n}Vol(M)\right)$ for all $t$.  Let $f$ be any minimizer of $\W(g(0), \cdot, \frac{1}{2})$.  Then by Proposition~\ref{mu properties} we have that 
\begin{equation*}
\left|Ric(t_{0}) + \nabla \overline{\nabla} f -\alpha -g(0)\right|^{2}=0.
\end{equation*}
By the definition of the Ricci potential, we must have
\begin{equation*}
\nabla \overline{\nabla} f  = \nabla \overline{\nabla} u,
\end{equation*}
and hence $ f= u+c$ for some constant $c$.  However, we clearly have that $
\tilde{f}= -n\log(2\pi) + \log( Vol(M) )$ is a minimizer of $\W(g(0),
\cdot, \frac{1}{2})$, which follows by direct computation.  As a result we
have that $u$ is a constant.  It follows from the normalizations that
$u=0$, and  $\omega_{0}$ is twisted K\"ahler-Einstein. 
\end{proof}

\section{The convergence of the twisted K\"ahler-Ricci flow}\label{sec: convergence}
\newcommand{\tKE}{g_{tKE}}

In this section we prove Theorem~\ref{thm:main}. The argument builds on
Tian-Zhu~\cite{TZ}, in that the behaviour of Perelman's entropy along the
flow is exploited. One important difference is that our distance function
$d(g)$ below measures the oscillation of the K\"ahler potentials rather
than a $C^{3,\alpha}$ norm of the metric. 

Suppose that $\tKE$ is a twisted K\"ahler-Einstein metric on $M$,
satisfying the equation
\begin{equation}\label{eq:tKE2}
	\mathrm{Ric}(\tKE) = \tKE + \alpha, 
\end{equation}
where $\alpha$ is a non-negative, closed, (1,1)-form. Write $G\subset
Aut_0(M)$ for the connected component of the group of biholomorphisms
preserving $\alpha$, i.e.
\[ G = \{ \tau\in Aut_0(M)\,:\, \tau^*(\alpha) = \alpha\}. \]
By Berndtsson's generalization~\cite{Ber} of Bando-Mabuchi's uniqueness
result~\cite{BM}, every solution of \eqref{eq:tKE2} is given by
$\tau^*\tKE$ for some $\tau\in G$. 

If $g = \tKE + i\partial\overline{\partial}\phi$, and $\tau\in G$, let
us define $\phi_\tau$ by
\[ \tau^*g = \tKE + i\partial\overline{\partial}\phi_\tau, \]
and let 
\[ d(g) = \inf\{ \mathrm{osc}\,\phi_\tau\,:\, \tau\in G\}. \]
Note that $d(g)$ is independent of the normalization of the K\"ahler
potentials.

Let us write $u_g$ for the twisted Ricci potential of $g$, normalized in any way
we like; i.e. 
\[ i\partial\overline{\partial} u_g = \omega + \alpha - \mathrm{Ric}(g). \]
Note that we can take $u_{\tau^*g} = \tau^*u_g$ for any $\tau\in G$. We will
work with $\osc u_g$, which is independent of the normalization,
and $\mathrm{osc}\, u_{\tau^*g} = \mathrm{osc}\, u_g$. 

The normalized twisted K\"ahler-Ricci flow \eqref{eq:twistedKRF}
is given by
\[ \frac{\partial}{\partial t} \phi(t) = u_{g(t)} + c(t), \]
where $c(t)$ is a time dependent constant depending on our
normalizations.

Theorem~\ref{thm:pman} implies that there is a constant $K$ depending on
$g(0)$, such that $\mathrm{osc}\, u_{g(t)} < K$ for all $t$. Moreover $K$
can be chosen uniformly as long as $g(0)$ is bounded in $C^3$ relative to
$\tKE$.

We will need the
following smoothing result for the twisted K\"ahler-Ricci flow. 
\begin{thm}\label{thm:smoothing}
	Suppose that $\omega(0) = \tKE +
	i\partial\overline{\partial}\phi(0)$
	for some fixed background metric $\tKE$,
	and $\osc\phi(0), \osc u(0) < K$ for some $K$. Then there exist $s,
	C > 0$ depending on $K$ (and $\tKE$), 
	such that at time $s$ along the twisted K\"ahler-Ricci 
	flow starting with
	$\omega(0)$, we have
	\[ \begin{gathered}
		C^{-1}\tKE < g(s) < C\tKE,\\
		\Vert g(s)\Vert_{C^3} < C, 
	\end{gathered}\]
	where the $C^{3}$ norm is measured using $\tKE$. 
\end{thm}
\begin{proof}
	This follows from Proposition 2.1 in
	Sz\'ekelyhidi-Tosatti~\cite{SzT}, and is similar to the result of
	Song-Tian~\cite{SongTian} for the K\"ahler-Ricci flow. 
	One just has to normalize $\phi$ and $u_g$ first 
	in order to bound $\sup|\phi(0)|$ and
	$\sup|u(0)|$.  
\end{proof}

In addition we will use the following result, which follows from the work of
Phong-Song-Sturm-Weinkove \cite{PSSW}.   
\begin{thm}\label{thm:PSSW}
	Suppose that along the twisted K\"ahler-Ricci 
	flow $g(t)$ we have $d(g(t)) < K$ for
	a constant independent of time. Then $g(t)$ converges to a
	twisted K\"ahler-Einstein metric exponentially fast. 
\end{thm}
\begin{proof}
	By Perelman's estimate we can assume that also $\osc u_{g(t)} <
	K$ for all $t$, increasing $K$ if necessary. 
	Fix a time $T$, and let $\tau\in G$ such that
	\[ \tau^*g(T) = \tKE + \ddb\phi, \]
	with $\osc\phi < K$. The Ricci potential still satisfies 
	$\osc u_{\tau^*g(T)} < K$, so the smoothing property applied to
	the twisted KR flow starting with $\tau^*g(T)$ implies that for some $s,
	C$ (depending on $K$) we have
	\[ \begin{gathered}
		C^{-1}\tKE < \tau^*g(T+s) < C\tKE \\
		\Vert \tau^*g(T+s) \Vert_{C^3(\tKE)} < C.
	\end{gathered}\]
	Since $T$ was arbitrary, this shows that up to the action of $G$,
	the metrics along the flow are uniformly bounded in $C^3$, relative
	to $\tKE$. It follows that a subsequence converges in
	$C^{2,\alpha}$, and the limit is necessarily a twisted
	K\"ahler-Einstein metric, since $u$ tends to a constant by
	Proposition~\ref{prop:  improved pman}. We can also obtain
	exponential convergence without needing the action of $G$, by
	following the argument of Phong-Song-Sturm-Weinkove \cite{PSSW}. Indeed 
	the uniform $C^3$ bound implies that the
	first eigenvalue of $\overline{\partial}$ on $TM$  (which is
	invariant under the action of $G$) is bounded away from zero
	uniformly. 
\end{proof}

We now show that the distance $d(g)$ is continuous with respect
to the $C^3$ metric on $g$. 
\begin{lem}\label{lem: d cont}
	If $g_k\to g$ in $C^3$ (with respect to $\tKE$), then $d(g_k)\to
	d(g)$. 
\end{lem}
\begin{proof}
	Recall that for $\tau\in G$ we wrote
	\[ \tau^*g = \tKE + \ddb\phi_\tau, \]
	and 
	\[ d(g) = \inf\{ \osc\phi_\tau\,:\, \tau\in G\}. \]
	We first prove that there exists a $\tau\in G$ realizing this
	infimum. Let $\tau_k\in G$ be a sequence so that $d(g) =
	\lim\osc\phi_{\tau_k}$. We can choose a constant $K$ such that
	\begin{equation}\label{ugK} \begin{gathered}
			\osc u_g < K \\
			\osc \phi_{\tau_k} < K.
		\end{gathered} 
	\end{equation}
	Using the smoothing property (applied to the twisted KR flow
	$\tau_k^*g(t)$) we can find $s, C$ such that
	\begin{equation}\label{gKE} 
			C^{-1} \tKE < \tau_k^*g(s) < C\tKE 
	\end{equation}
	Fix a point $p\in M$. Choosing a subsequence of the $\tau_k$ we can
	assume that $\tau_k(p) \to q$ for some $q\in M$. From \eqref{gKE}
	it follows that for large enough $k$, each $\tau_k$ maps an open
	coordinate neighborhood $B_p$ about $p$ to a coordinate neighborhood
	about $q$. The component functions of these $\tau_k$ are then given
	by uniformly bounded holomorphic functions on $B_p$, so after
	choosing a further subsequence, we can assume that the $\tau_k$
	converge when restricted to the half ball $\frac{1}{2}B_p$. 
	The open sets $\frac{1}{2}B_p$ cover $M$,
	so we can choose a finite subcover, and a subsequence of the
	$\tau_k$ will then converge over all of $M$ to a holomorphic map
	$\tau_\infty: M\to M$. Taking the limit in \eqref{gKE} we see that $\tau$
	is injective, and it is an open map so it is also surjective.
	Moreover, $\tau$ clearly preserves $\alpha$, and
	 the connected component of the identity in $Aut(M)$ is
	closed, so $\tau_\infty\in G$. It follows that $d(g) =
	\osc\phi_{\tau_\infty}$. 	

	Suppose now that $g_k\to g$ in $C^3$, and $\tau\in G$ realizes the
	infimum $d(g)$. Since $\tau^*g_k\to \tau^*g$, using the same $\tau$
	to bound each $d(g_k)$ we find that
	\[ \lim\sup d(g_k) \leqslant d(g). \]
	For the converse inequality suppose that for each $k$, $\tau_k\in
	G$ realizes the infimum $d(g_k)$. By the same argument as above (we
	can choose a uniform $K$ in \eqref{ugK} for all the $g_k$), up
	to choosing a subsequence we can assume that
	$\tau_k\to\tau_\infty$. Then $\tau_k^*g_k \to \tau_\infty^*g$, and
	using $\tau_\infty$ to bound $d(g)$ we get
	\[ d(g) \leqslant \lim\inf d(g_k). \]
	This shows that $d(g) = \lim d(g_k)$. 
\end{proof}

We will now write $\mu(g) = \mu\left(g, \frac{1}{2}\right)$ for Perelman's
entropy. 
We collect here a few of the previous results about $\mu$, from
Lemma~\ref{lem:  mu upper bound} and Corollary~\ref{cor:mu max}:
\begin{enumerate}
	\item $\mu$ is continuous in the $C^3$-norm, measured
		relative to $\tKE$.   
	\item For any initial metric $g(0)$, $\mu(g(t))$ is
		monotonically increasing along the twisted K\"ahler-Ricci 
		flow $g(t)$, and
		$\lim\mu(g(t)) = \Lambda$, where we fix
		$\Lambda = \log\left( (2\pi)^{-n} Vol(M) \right)$.  In particular, $\Lambda$ is 
		independent of $g(0)$. 
	\item $\mu(g) = \Lambda$ if and only if $g$ is a twisted KE metric.  By Berndtsson's 
		uniqueness theorem this is equivalent to: $\mu(g) = \Lambda$ if and only if $g = \tau^{*}\tKE$ 
		for a biholomorphism $\tau\in G$ of $(M,J)$ fixing $\alpha$. 
\end{enumerate}

The following is a consequence of the smoothing result. 
\begin{lem}\label{lem: short time d}
	Fix $K > 0$, and suppose that 
	\[ \begin{gathered}
		1 \leqslant d(g)  < K \\
		\osc u_{g(t)} <K, 
	\end{gathered}
	\]
	for all $t$ along the twisted KR flow $g(t)$ with initial metric $g$. 
	There exist $s, C >  0$ 
	depending on $K$, and a $\tau\in G$ such that
	\[ \begin{gathered}
		d(g(s)) \geqslant \frac{1}{2}, \\
		C^{-1}\tKE < \tau^*g(s) < C\tKE,\\ 
		\Vert \tau^*g(s)\Vert_{C^3} < C,
	\end{gathered}\]
	where the $C^3$-norm is measured with respect to the fixed
	metric $\tKE$.  
\end{lem}
\begin{proof}
	For any $\tau\in G$, the flow $\tau^*g(t)$ is a solution of the
	twisted KR flow, and it follows by our assumption that 
	$\osc u_{\tau^*g(t)} < K$ along the flow with initial metric $\tau^*g$. Write
	\[ \tau^*g(t) = \tKE + i\partial\overline{\partial}\phi(t), \]
	so that $\mathrm{osc}\,\phi(0) \geqslant 1$ by the assumption that
	$d(g)\geqslant 1$. 
	Along the twisted KR flow, 
	\[ \dot\phi(t) = u_{\tau^*g(t)} + c(t), \]
	where $c(t)$ is a time dependent constant. So if $s <
	(4K)^{-1}$ we have
	\[  \begin{gathered}
	\sup\phi(s) > \sup\phi(0) - Ks + A> \sup\phi(0) - \frac{1}{4} +A \\
	\inf\phi(s) < \inf\phi(0) + Ks + A< \inf\phi(0) + \frac{1}{4} +A, 
	\end{gathered} \]
	for some constant $A$ (the integral of $c(t)$), and so 
	\[ \mathrm{osc}\,\phi(s) > \mathrm{osc}\,\phi(0) - \frac{1}{2}
		\geqslant \frac{1}{2}. 
	\]
	Since this is true for any $\tau\in G$, by taking infimum we have
	\[ d(g(s)) \geqslant \frac{1}{2}. \]
	For the smoothing result, we first choose a $\tau\in G$ such
	that
	\[ \tau^*g = \tKE + \ddb\phi, \]
	with $\osc\phi < K$. Then we can use the smoothing theorem,
	applied to the flow $\tau^*g(t)$ with initial metric $\tau^*g$,
	and choose $s$ even smaller than we did in the previous step
	if necessary. 
\end{proof}

\begin{prop}\label{prop:no escape}
	Fix $K > 0$. There
	exists a $c > 0$ depending on $K$, such that if 
	\[ \begin{gathered}
		\mu(g) > \Lambda - c,\\
		d(g), \osc u_{g(t)} < K,
	\end{gathered}\]
	for all time $t$ along the twisted KR flow $g(t)$, then $d(g) < 1$. 
\end{prop}
\begin{proof}
	We argue by contradiction. Suppose there is a $K > 0$
	for which there is no suitable $c$. This means that
	we can choose a sequence $g^k$ for which
	\begin{equation}\label{eq1}
		\mu(g^k) > \Lambda - 1/k, 
	\end{equation}
	and $d(g^k), \osc u_{g_k(t)} < K$ for all $t$, but
	$d(g^k) \geqslant 1$.

	We apply Lemma ~\ref{lem: short time d}. We get $s, C > 0$, and a $\tau_k\in G$
	such that
	\begin{equation}\label{1}
		\begin{gathered}
		d(g^k(s)) \geqslant \frac{1}{2} \\
		C^{-1}\tKE < \tau_k^*g^k(s) < C\tKE,\\
		\Vert \tau_k^*g^k(s)\Vert_{C^3(\tKE)} < C,\\ 
		\mu(\tau_k^*g^k(s)) = 
		\mu(g^k(s)) \geqslant \mu(g^k) > \Lambda - 1/k,
		\end{gathered}
	\end{equation}
	where we also used the monotonicity of $\mu$ along the twisted
	KR flow.
	We can choose a subsequence of the $\tau_k^*g^k(s)$ converging to some
	$g_\infty$ in $C^{2,\alpha}$ and in particular
	$\mu(g_{\infty}) = \Lambda$, so $g_\infty = \tau^*\tKE$ for
	some $\tau\in G$.  The metric $g_\infty$ is uniformly equivalent to
	$\tKE$ (we can use the constant $C$ given by the bounds
	\eqref{1}) and so 
	\[\begin{aligned}
		\Vert (\tau^{-1})^*\tau_k^*g^k(s) -
		\tKE\Vert_{C^{2,\alpha}(\tKE)} &= 
		\Vert (\tau^{-1})^*\tau_k^*g^k(s) -
		(\tau^{-1})^*g_\infty\Vert_{C^{2,\alpha}(\tKE)} \\
		&= \Vert
		\tau_k^*g^k(s) - g_\infty\Vert_{C^{2,\alpha}(\tau^*\tKE)} \to 0.
	\end{aligned}\]
	Writing
	\[ (\tau^{-1})^*\tau_k^*g^k(s) = \tKE +
		i\partial\overline{\partial}\phi_k(s)\] 
	with $\phi_k(s)$ normalized to have
	zero mean, we obtain
	\[ \phi_k(s) \to 0\,\text{ in }C^{4,\alpha}(\tKE), \]
	which contradicts 
	\[ \mathrm{osc}(\phi_k(s)) \geqslant d(g^k(s)) \geqslant
		\frac{1}{2}.\] 
\end{proof}

Next we prove a stability result for the twisted K\"ahler-Ricci flow.  In the untwisted case, a similar result 
was proved by Sun-Wang~\cite{SW}, using different techniques.
\begin{prop}\label{prop:stability}
	There is a $\delta > 0$ such that if $\Vert g -
	\tKE\Vert_{C^3(\tKE)}< \delta$, then the twisted K\"ahler-Ricci
	flow starting with $g$ converges to a twisted 
	K\"ahler-Einstein metric. 
\end{prop}
\begin{proof}
	We can choose a constant $K > 1$ such that if 
	$\Vert g - \tKE\Vert_{C^3(\tKE)} < \frac{1}{2}$, 
	then $d(g) < K$, and also 
	by Perelman's estimate $\osc u_{g(t)} <K$ for all $t$.  
	Let $c=c(K)$ be the constant given by Proposition~\ref{prop:no escape}. 
	We can now choose $\delta < 1$ sufficiently small
	so that
	\[ \Vert g-\tKE\Vert_{C^3(\tKE)} < \delta \]
	implies that $\mu(g) > \Lambda - c$. 
	By the monotonicity we have
	\[ \mu(g(t)) > \Lambda - c \]
	for all $t > 0$. The previous Proposition implies that $d(g) < 1$ and
	then also $d(g(t)) < 1$
	for all $t > 0$ (since at any time if $d(g(t)) = 1$, then Proposition~\ref{prop:no escape}
	 says $d(g(t)) < 1$). Theorem \ref{thm:PSSW} implies
	that $g(t)$ converges to a twisted KE metric. 
\end{proof}

Let $S$ be the set of $C^{5}$ metrics $g$ in the class $[\tKE]$ such that the
twisted KR flow starting with $g$ converges to a twisted KE metric in $C^{5}(\tKE)$. 
Our goal is to prove that $S$ is both open and closed.  We claim that this implies our main theorem.
To see this, observe that the set of $C^{5}$ metrics  in the class $[\tKE]$ is convex, and hence connected.  
It follows that either $S$ contains all metrics,
or $S$ is empty.  But $\tKE \in S$, and hence Theorem~\ref{thm:main} follows.
\begin{prop}
	$S$ is open in the $C^5$ topology. 
\end{prop}
\begin{proof}
	Suppose that $g\in S$. Then for sufficiently large $T$ there exists
	a $\tau\in G$ such that
	$\Vert \tau^*g(T) - \tKE\Vert_{C^3(\tKE)} < \delta/ 2$ 
	with the $\delta$ from the previous result.
	For a finite time $t\in [0,T]$ the solution of the twisted KR flow depends
	smoothly on the initial data, so, since $\tau \in G$ is fixed, we can choose $c$ small
	so that if $\Vert h-g\Vert_{C^{5}(\tKE)} < c$, then $\Vert
	\tau^*h(T)-\tau^*g(T)\Vert_{C^{3}(\tKE)} < \delta/2$. Then 
	\[ \Vert \tau^*h(T) - \tKE\Vert_{C^3(\tKE)} < \delta, \]
	so the stability result implies that the flow starting with
	$\tau^*h(T)$, and hence also the flow starting with $h$, 
	converges to a twisted K\"ahler-Einstein metric. 
\end{proof}

\begin{prop}\label{prop:closed}
	$S$ is closed in $C^5$. 
\end{prop}
\begin{proof}
	Suppose that $g_k\in S$, and
	\[ g_k \to g \text{ in } C^5. \]
	By Perelman's estimate, we can choose a $K >2$ such that
	$\osc u_{g_k(t)} < K$ for all $k, t$ (using that the $g_k$ are in a
	bounded set of metrics in $C^5$). Let $c$ be the constant given
	by Proposition~\ref{prop:no escape} corresponding to $K$. By the properties of
	$\mu$, there exists a $T$ such that
	\[ \mu(g(T)) > \Lambda - c. \]
	Since the tKRF is stable for finite time, and $\mu$ is continuous for 
	a $C^{3}$ family
	of metrics, there exists an $N$ such that 
	\[ \mu(g_k(T)) > \Lambda - c\]
	for all $k > N$. By monotonicity, it follows that
	\[ \mu(g_k(t)) > \Lambda - c\]
	for all $k > N$ and $t \geqslant T$. Since we know that $g_k(t)$
	converges to a twisted KE metric $\tau_k^*\tKE$ for some $\tau_k\in G$,
	it follows that $(\tau_{k}^{-1})^*g_k(t)$ converges to $\tKE$. Now, since 
	$(\tau_{k}^{-1})^*g_k(t)$ converges to $\tKE$
	we can apply Lemma~\ref{lem: d cont} to obtain that $d((\tau_{k}^{-1})^*g_k(t)) = 
	d(g_k(t))$ converges to $d(\tKE)=0$.  In particular,  
	there must be a first time $t_k\geqslant T$ for which 
	\[ d(g_k(t_k)) \leqslant K/2. \]
	By Proposition~\ref{prop:no escape}, at this time we have $d(g_k(t_k)) < 1$, so
	since $K > 2$ and $t_k\geqslant T$ was chosen to be minimal, 
	$t_k=T$. But then as before, we have $d(g_k(t)) < 1$ 
	for all $t > T$ and $k >
	N$. For any fixed $t$ we have $g_k(t)\to g(t)$ in $C^3(\tKE)$ 
	as $k\to\infty$, and so apply Lemma~\ref{lem: d cont} again, we obtain
	\[ d(g(t)) \leqslant 1\]
	for all $t \geqslant T$. Theorem \ref{thm:PSSW} then implies
	that $g(t)$ converges to a twisted KE metric. 
\end{proof}

\end{document}